\documentclass{amsart}
\usepackage[utf8]{inputenc}
\usepackage[english]{babel}
\usepackage{amsmath,amssymb,amsthm,amsfonts,graphicx,tikz,fullpage}
\usepackage{pdfpages}

\usepackage{array}

\usepackage{mathtools}

\theoremstyle{definition}
\newtheorem{thm}{Theorem}[section]
\newtheorem{prop}[thm]{Proposition}

\newtheorem{conj}[thm]{Conjecture}
\newtheorem{defn}[thm]{Definition}

\newtheorem{lem}[thm]{Lemma}

\newtheorem{lemma}[thm]{Lemma}

\newtheorem{example}[thm]{Example}

\DeclareMathOperator{\glex}{glex}
\DeclareMathOperator{\grevlex}{grevlex}
\DeclareMathOperator{\lex}{lex}

\DeclareMathOperator{\init}{in}

\DeclareMathOperator{\supp}{supp}
\DeclareMathOperator{\wt}{wt}

    \newcommand{\calA}{\mathcal{A}}    
    \newcommand{\calB}{\mathcal{B}}
    \newcommand{\cC}{\mathcal{C}} 
    \newcommand{\cD}{\mathcal{D}}
    \newcommand{\cE}{\mathcal{E}}
    \newcommand{\cF}{\mathcal{F}}
    \newcommand{\cG}{\mathcal{G}}
    \newcommand{\cL}{\mathcal{L}}
    \newcommand{\cU}{\mathcal{U}}
    \newcommand{\cV}{\mathcal{V}}
    \newcommand{\cW}{\mathcal{W}}
    \newcommand{\cZ}{\mathcal{Z}}
    \newcommand{\CC}{\mathbb{C}}
    \newcommand{\commentout}[1]{}

    \newcommand{\iso}{\cong}
    
    \newcommand{\lglex}{\prec_{\glex}}
    \newcommand{\lgrevlex}{\prec_{\grevlex}}

    \newcommand{\bbC}{\mathbb{C}}
    \newcommand{\NN}{\mathbb{N}}

    \newcommand{\RR}{\mathbb{R}}
    \newcommand{\bbR}{\mathbb{R}}
    \newcommand{\ZZ}{\mathbb{Z}}
    \newcommand{\bbZ}{{\mathbb{Z}}}
    \newcommand{\initp}{\init_{\prec}}
    
    \newcommand{\wtil}{\widetilde}

    \newlength{\dividendlength}
    \newlength{\divisorlength}
    \newlength{\dividendheight}
    \newlength{\divisorheight}
    \newlength{\maxheight}
    
\begin{document}
    
    \title{Universal Gr\"obner Bases of Toric Ideals of Combinatorial Neural Codes}
    
    \author{Melissa Beer}
    \address{
        Department of Mathematics and Computing\\
        Franklin College\\
        Franklin, IN 46131}
    \email{Melissa.Beer@franklincollege.edu}
    
    \author{Robert Davis}
    \address{
        Department of Mathematics\\
        Harvey Mudd College\\
        Claremont, CA 91711}
    \email{rdavis@hmc.edu}
    
    \author{Thomas Elgin}
    \address{
        Department of Mathematics\\
        University of South Carolina\\
        Columbia, SC 29208}
    \email{thomaselgin97@gmail.com}
    
    \author{Matthew Hertel}
    \address{
        Department of Mathematics\\
        Michigan State University\\
        East Lansing, MI 48824}
    \email{hertelma@msu.edu}
    
    \author{Kira Laws}
    \address{
        Department of Mathematical Sciences\\
        Appalachian State University\\
        Boone, NC 28608}
    \email{lawskn@appstate.edu}
    
    \author{Rajinder Mavi}
    \address{
        Department of Mathematics\\
        Ripon College\\
        Ripon, WI 54971}
    \email{mavir@ripon.edu}
    
    \author{Paula Mercurio}
    \address{
        Department of Mathematics\\
        Michigan State University \\
        East Lansing, MI, 48824}
    \email{mercuri6@msu.edu}
    
    \author{Alexandra Newlon}
    \address{
        Department of Mathematics\\
        Colgate University \\
        Hamilton, NY, 13346}
    \email{anewlon@colgate.edu}
    
    \thanks{
    This work was partially supported by the National Security Agency (NSA Award No. H98230-18-1-0042), the National Science Foundation (NSF Award No. 1559776). 
    The authors would like to thank Robert W. Bell for his work in organizing the 2018 SURIEM program, during which this work was developed.}
    \maketitle
    
    \begin{abstract}
        In the $1970$s, O’Keefe and Dostrovsky discovered that certain neurons, called place cells, in an animal’s brain are tied to its location within its arena.
        A combinatorial neural code is a collection of $0/1$-vectors which encode the patterns of co-firing activity among the place cells.
        Gross, Obatake, and Youngs have recently used techniques from toric algebra to study when a neural code is $0$- $1$-, or $2$-inductively pierced: a property that allows one to reconstruct a Venn diagram-like planar figure that acts as a geometric schematic for the neural co-firing patterns.
        This article examines their work closely by focusing on a variety of classes of combinatorial neural codes. 
        In particular, we identify universal Gr\"obner bases of the toric ideal for these codes.
    \end{abstract}
    
    \section{Introduction}
    
    A \emph{combinatorial neural code} is a set of $0/1$-vectors that is used to model the co-firing patterns of certain neurons in the brain of an animal.
    These neurons are call \emph{place cells} and are active when the animal is in a particular region within its environment, called a \emph{place field}, or simply just a \emph{field}.
    Here, we are not concerned with timing and spiking of neural activity; we consider only the case where the place cell is considered ``on'' or ``off.''
    
    Recall that \emph{Venn diagram} is a diagram consisting of regions bounded by $n$ simple, closed curves such that all possible intersections of the curves' interiors appear.
    An \emph{Euler diagram} is a generalization of a Venn diagram, where the curves' interiors do not need to intersect in all possible ways. 
    Consider the following Euler diagram, where $U_i$ refers to the interior of the (innermost) curve in which the label is contained:
    \begin{center}
    	\begin{tikzpicture}
    		\draw (-3,0.5) -- (3,0.5) -- (3,4.5) -- (-3,4.5) -- cycle;
    		\draw (-1,2.5) circle [radius=1.5];
    		\draw (-0.75,2.5) circle [radius=.75];
    		\draw (1,2.5) circle [radius=1.5];
    			
    		\node at (-1.9,3.2) {$U_1$};
    		\node at (-1.1,2.5) {$U_2$};
    		\node at (2,2.5) {$U_3$};
    	\end{tikzpicture}
    \end{center}
    This diagram models the co-firing patterns of three place cells, and the regions in which the place cells are active are inside of the three circles.
    The labels $U_1$, $U_2$, and $U_3$ are the interiors of the three curves, and the regions of the diagram can be encoded by triples of $0$s and $1$s, indicating which of the place cells are active.
    We use the standard convention of using $0$ to denote an inactive neuron and $1$ to denote an active neuron.
    So, for example, the codeword $101$ corresponds to the intersection $U_1 \cap U_2^c \cap U_3$, where $U_2^c$ denotes the complement of $U_2$ in the diagram.
    The full neural code is
    \begin{equation}\label{ci}
    	\cC = \{000, 100, 001, 110, 101, 111\}.
    \end{equation}
    
    \commentout{We will be interested in rings of polynomials which are generated by such codes. For $\cC$ defined in \eqref{ci} the associated ring is 
    \[  R_1 = \CC[x,z,xy,xz,xyz].  \]
    (note the exclusion of $1 = x^0y^0z^0$).
    Using this notation, $R_1$ is the closure of $x,z,xy,xz,xyz$ under products and sums.
    Examples of some elements are
    \[ 1,\, x, \, x+z, \,x + xy, \dots , \, x^5,\dots, \, 4+ 3(xy)^6 + (xz)^2, \textnormal{ etc}\dots \in R_1 \]
    
    In fact, $R_1$ can be written more succinctly:
    \[  \CC[x,z,xy] \cong R_1  \]
    since there is no need to write $xz$ if we have both $x$ and $z$.
    If we label these generators as $t_1 = x$, $t_2 = z$, and $t_3 = xz$, then the relationship among these generators, called a \emph{syzygy}, is $t_1t_2 = t_3$.
    It is precisely the syzygies of these rings, we will be interested in later.
    
    \begin{defn}
    	A set $S \subseteq \RR^n$ is \emph{convex} if, for every $u,v \in S$, the line segment
    	\[
    		\{ tu + (1-t)v \mid 0 \leq t \leq 1\}
    	\]
    	is contained entirely in $S$.
    \end{defn}
    }
    
    Intuitively, convex sets do not have holes or dents; a disc is convex, but a circle is not.
    If a code has an Euler diagram consisting of convex sets, then the code is called \emph{convexly realizable}.
    More generally, one may ask whether every code is convexly realizable in $\RR^n$ for some $n$ rather than just $\RR^2$. 
    The answer is yes \cite{FrankeMuthiah}, but it is much less clear how to determine the smallest $n$ needed.
    In this article, we focus on convex realizability in $\RR^2$ only, turning neural codes into algebraic objects, with the goal of deducing information about the code via algebraic techniques.
    
    \subsection{Models and results}

     An Euler diagram $\cD$ is a collection of sets $\cD = \{U_1,...,U_n\}$, which we refer to as place fields, and where each field  $ U_i$ is a subset of $ \bbR^2$. The sets $U_i$ are sufficiently nice, with boundaries $\lambda_i = \partial U_i $ which are piecewise smooth curves. We will label the set of boundary curves $\Lambda = \{\lambda_1,..,\lambda_n\}$. Our next step is to label the connected components of $\bbR^2 \setminus(\lambda_1 \cup \cdots \cup \lambda_n)$. Thus, for any codeword $w \in \{0,1\}^{[n]}$
     the associated \emph{zone} is defined as 
     \[ Z_w = \left( \bigcap_{i:w_i = 1}  U_i \right) \cap \left( \bigcap_{i:w_i = 0}  U_i^c \right).  \]
     The code for $\cD$ is the set
     \[ \cC_{\cD} = \{  w \in \{0,1\}^{[n]}  : Z_w \neq \emptyset \}  \]
     and the zones are collected in the set $\cZ_{\cD} = \{ Z_w : w \in \cC \}$.

    \begin{defn}\label{wf}
        An Euler diagram is {\it well-formed} if
        \begin{enumerate}
            \item each curve label is used exactly once,
            \item all curves intersect at exactly  0 or 2 points,
            \item each point in the plane is passed by at most 2 curves, and 
            \item each zone  is connected.
        \end{enumerate}
        If $\cC$ is a code with a well-formed Euler diagram, then we call $\cC$ a \emph{well-formed code}.
    \end{defn}
    
    Requiring a diagram to be well-formed can be partly thought of as insisting that the curves in the diagram intersect ``generically enough,'' as long as the zones stay connected.
    It follows from Definition \ref{wf} that the zones are the exactly the connected components of $\RR^2 \setminus (\lambda_1\cup \cdots \cup \lambda_n)$. 
    
    \subsubsection{Special types of diagrams}
    
    In this work, we are able to draw conclusions from diagrams of limited complexity. This is specified by allowing diagrams of limited {\it depth} or diagrams constructed of {\it zero-} or {\it one-piercings}.
    
    \begin{defn}
        Let $\cD$ be a well-formed Euler diagram.
       A curve $\lambda = \partial U$ is a \emph{$0$-piercing} of $\cD$ if,  for all $i \in [n]$, $U \subset U_i$, $U \supset U_i$, or $U \cap U_i = \emptyset$.
       A curve $\lambda = \partial U$ is a \emph{$1$-piercing} of $\cD$ if there is exactly one $j \in [n]$ so that the sets $U\cap \lambda_j $, $U \setminus U_j$, and $U_j \setminus U $ are nonempty.
      
       We say $\cD$ is \emph{$0$-inductively pierced} if there exists some labeling of the curves $\lambda_1,\lambda_2,\dots, \lambda_n$ so that for each $k \in [n-1]$, $\lambda_{k+1}$ is a $0$-piercing of the diagram $\cD_k = \{U_1,..,U_k\}$. Similarly, a diagram $\cD$ is \emph{$1$-inductively pierced} if there is a labeling of the curves so that for each $k \in [n-1]$, $\lambda_{k+1}$ is a $0$- or $1$-piercing of the diagram $\cD_k$.
    \end{defn}
    
    These $0$- and $1$-inductrively pierced diagrams are special cases of $k-$inductively pierced diagrams, where $k$ may be any nonnegative integer. 
    In this paper we focus only on the $k = 0,1$ cases but we refer to \cite{TheoryPiercings} for further background on inductively pierced codes.
    We now define the depth of a diagram, which will be an important factor in the results we present.
    
    \begin{defn}
        A field $ U \in \cD $ is of \emph{} $d$ if there are $U_{i_1},..,U_{i_d} \in \cD$ such that 
      \[U \subset U_{i_1}  \subset \cdots \subset U_{i_d}.  \]
      The \emph{depth} of the diagram $\cD$ is the maximum depth over all fields in the diagram. 
    \end{defn}

    \subsubsection{Algebraic construction}

    We will study homomorphisms between polynomial rings generated by variables respectively labeled by the diagram's zones $\cZ$ and the diagram's field labels $\Lambda$. 
    
    For each $w \in \cC$ we identify $w$ with the map $[n] \mapsto \{0,1\}^n$ that sends $i$ to $0$ if $w_i = 0$ and sends $i$ to $1$ otherwise.
    Additionally, we label a variable associated to the zone $Z_w $ by $ t_{ w^{-1}(1) } $. For each $ \lambda \in \Lambda  $ we label a variable associated to the field by $x_\lambda$. We introduce a homomorphism 
    \[ \pi_{\cC}: \bbC[ t_{ w^{-1}(1) } : w \in \cC ] \to  \bbC[ x_{ \lambda } : \lambda \in \Lambda ]  \]
     via the mapping
     \[ \pi_{\cC}(t_{w^{-1}(1)}) \mapsto x^{w} = \prod_\lambda x_\lambda^{w_\lambda}  \]
    The primary object we are interested in is the \emph{toric ideal} $I_{\cC}$, defined as the kernel of the map $\pi_{\cC}$. Hilbert's Basis Theorem guarantees that $I$ is finitely generated. Moreover, it is not hard to see that the ideal is generated by binomials.  
    We will be especially interested in a set of binomial generators with particular properties.
    
    To define the special generating set, we first note that  monomial order on $\CC[t_1,\dots,t_m]$ is a total ordering $\prec$ of its monomials such that
    \begin{enumerate}
    	\item $\prec$ respects multiplication: if $u,v,w$ are monomials and $u \prec v$, then $uw \prec vw$, and
    	\item $\prec$ is a well-ordering: $1 \prec u$ for all monomials $u$.
    \end{enumerate}
    
    As a first example, we describe a well-known and computationally-efficient order.
    The \emph{graded reverse lexicographic order}, or simply \emph{grevlex} order, on $\CC[t_1,\dots,t_m]$ is denoted by $\lgrevlex$ and sets $t^a \lgrevlex t^b$ if $\sum a_i < \sum b_j$ or if $\sum a_i = \sum b_j$ and the last nonzero entry of $a-b$ is negative.
    While this is not the most intuitive monomial order, it will still be useful to us.
    
    A second useful class of monomial orders on $\CC[t_1,\dots,t_n]$ is the set of weight orders. A \emph{weight order} $\prec_{w,\sigma}$ is determined by a vector $w \in \RR^n$ and an existing monomial order $\prec_\sigma$, and sets $t^a \prec_{w,\sigma} t^b$ if and only if either $w\cdot a < w \cdot b$, where $\cdot$ is the dot product, or $w\cdot a = w \cdot b$ and $t^a \prec_{\sigma} t^b$; for this reason, $\prec_\sigma$ is often informally referred to as the ``tie-breaker.''
    
    Now, given a monomial ordering $\prec$, any polynomial $f$ has a unique initial term which is denoted $\initp(f)$. This further leads to the initial ideal of a given ideal $I$ defined as 
    \[ \initp(I) = \{ \initp(f) : f \in I\}. \]
    If $I = (g_1,..,g_k)$, it is not necessarily true that $(\initp(g_1),..., \initp(g_k))$ equals $\init(I)$. 
    
    \begin{defn}
      Let $\cG = \{g_1,\dots,g_k\}$ be a generating set of an ideal $I$ of $\CC[x_1,\dots,x_n]$ and let $\prec$ be a monomial order.
      If
      \[
        \initp(I) = (\initp(g_1),\dots,\initp(g_k)
      \]
      then we call $\cG$ a  \emph{Gr\"obner basis} of $I$ with respect to $\prec$. 
      Moreover, we say a $\cG$ is \emph{reduced} if the leading coefficient (with respect to $\prec$) of every element is $1$ and if, for every $g,g' \in \cG$, $\initp(g)$ does not divide any term of $g'$.
    \end{defn}
    
    Although there might be  many Gr\"obner bases for a given ideal and monomial order, there is a unique reduced Gr\"obner basis. Finally, we say that a Gr\"obner basis $\cG$ is a \emph{universal Gr\"obner basis} if it is a Gr\"obner basis with respect to any monomial order. Because $I$ has finitely many initial ideals, we can always construct a finite universal Gr\"obner basis by taking the union of all reduced Gr\"obner bases of $I$. We will call this union {\it the} universal Gr\"obner basis. 
    
    In \cite{GrossObatakeYoungs}, the authors successfully found ways to algebraically detect when a code is $k$-inductively pierced for small $k$ and/or few neurons.
    The main results of that article are summarized in the theorem below.
    
    \begin{thm}[see \cite{GrossObatakeYoungs}]\label{thm: GOY}
    	Let $\cC$ be a well-formed code on $n$ neurons such that each neuron fires at least once, that is, there is no $i$ for which $w_i = 0$ for all codes $w \in \cC$.
    	\begin{enumerate}
    		\item The toric ideal $I_{\cC} = (0)$ if and only if $\cC$ is $0$-inductively pierced.
    		\item A well-formed diagram $\cD$ is inductively $0$-pierced if and only if no two curves in any well-formed realization of $\cD$ intersect.
    		\item If $\cC$ is $1$-inductively pierced then the toric ideal $I_{\cC}$ is either generated by quadratics or $I_{\cC} = (0)$.
    		\item When $n=3$, the code is $1$-inductively pierced if and only if the reduced Gr\"obner basis of $I_{\cC}$ with respect to the weighted grevlex order with the weight vector $w = (0, 0, 0, 1, 1, 1, 0)$ contains only binomials of degree $2$ or less.
    	\end{enumerate}
    \end{thm}
    
    The authors further made the following conjecture.
    
    \begin{conj}[see \cite{GrossObatakeYoungs}]
    	For each $n$, there exists a monomial order such that a code is $0$- or $1$-inductively pierced if and only if the reduced Gr\"obner basis contains binomials of degree $2$ or less.
    \end{conj}

    \subsubsection{Graphical construction}
    
      Any Euler diagram can be associated to a graph. Recall a graph is a pair $(\cV,\cE) $ where  $\cV$  is a set of vertices and $\cE$ is a set of two-element subsets of $\cV$. 
    \begin{defn}[Dual graphs]
        Given a code $\cC$ with $m = |\cC|$ elements, we define a \emph{dual graph} $G_\cC$ whose vertices are labeled uniquely by elements of the code $\cC$. A pair $\{w_1,w_2\}$ is an edge if and only if zones $Z_{w_1}$ and $Z_{w_2}$ have a nontrivial intersection of the boundary, i.e. $\partial Z_{w_1}\cap \partial Z_{w_2}\neq \emptyset $. For our purposes, these dual graphs we will always include a vertex labeled with the codeword $0\dots0$.
    \end{defn}
    
    Furthermore, we can define the obvious inclusion mapping $\iota:\cC \xhookrightarrow{}\bbZ^n$.  Then a weight function (a distinct notion from weighted monomial orders) can be introduced by setting
    \begin{equation}\label{mudef}
         \mu(  w ) = \| \iota\circ w\|_1.  
    \end{equation}  
    \begin{defn}
        [Weighted  dual graphs]
        A \emph{weighted dual graph} is a triple $(\cV,\cE, \mu)$ such that $(\cV,\cE)$ is a dual graph and $\mu$ is a mapping as in \eqref{mudef}.
    \end{defn} 
    Notice an edge between two nodes $w_1$ and $w_2$ exists only if
    \[  \| \iota\circ w_1 - \iota\circ w_2 \|_1 = 1 . \]
    With some abuse of notation, we can extend the definition of $\mu$ to monomials:
    \begin{equation}\label{mudef ext}
        \mu\left(\prod_{w \in \cC}t_{w^{-1}(1)}^{a_w} \right) 
        = \left\| \sum_{w \in \cC} a_w \iota \circ w  \right\|_1. 
    \end{equation}     
    We say a binomial is of \emph{weight} $k$ if all the terms are of weight $k$.
    Determining the binomials in the toric ideal of a code from its Euler diagram can be reduced to finding specific subgraphs of  the dual graph. We will now define subgraph embeddings. 
    
    \begin{defn}
        [Weighted graph embeddings] We say the weighted dual graph $H = (\cV,\cE,\mu)$ is \emph{embedded} in the weighted dual graph $G = (\cW, \cF, \nu)$ if there is a one-to-one mapping $\phi: \cV \to \cW$ such that $\mu(x)= \nu(\phi(x))$ for all $x \in \cV$ and the pair $\{\phi(x),\phi(y)\}$ belongs to $\cF$ for all $\{x,y\}\in \cE$.
    \end{defn}
    
    Notice we can extend any graph embedding $\phi$ to a map on the associated polynomial rings. Explicitly, let $H$ be a graph constructed from the neural code $\wtil \cC$ which  embeds via $\phi$ into a graph $G$ constructed from the neural code $\cC$. Let $\jmath$ (respectively $\wtil \jmath$) map codewords in $\cC$ (respectively $\wtil \cC$) to nodes of the graph $G$ (respectively $H$). Define the mapping
    \[ \phi:  \bbC[ t_{ v^{-1}(1) } : v \in \wtil \cC ] \to     \bbC[ t_{ w^{-1}(1) } : w \in  \cC ]     \]    
    such that $ \phi(t_{v^{-1}(1)}) = t_{w^{-1}(1)} $ if and only if $ \phi (\wtil (\jmath) v) = \jmath ( w ) $.

    \subsection{Results}
    
    Here we will summarize the main results of this article.
    First, we give two more definitions. 
    
    \begin{defn}
        Given a code $C$ let
        \[ \calA_\cC 
        = \{ t_{w^{-1}(c_1)} t_{w^{-1}(c_2)} - t_{w^{-1}(c_3)} | c_1,c_2,c_3\in \cC : c_1 + c_2  = c_3; \mu(c_1) = \mu(c_2) = 1 \} \]
    \end{defn}
    \begin{defn}
    A code $\cC \subset \{0,1\}^{[n]}$ is \emph{external} if there are $n$ unique codewords $w_i$ such that $\mu(w_i) = 1$.
    \end{defn}
    
    \begin{thm}\label{Thm Aext}
        For an external code $\cC$, the indispensable binomials (defined in the next section) of $I_\cC$ are exactly those in $\calA_\cC$.
    \end{thm}
      
     In the last section we define a class of codes called \emph{internal codes}. 
     For these, we are able to find generating sets that are Gr\"obner bases for all term orders. 
     In fact, these \emph{universal Gr\"obner bases} consist entirely of quadratic binomials.

     Before we proceed to the proof of Theorem \ref{Thm Aext}, we will examine 1-inductively pierced diagrams of maximum depth 1 in Section \ref{depth1}. For such diagrams we can identify exactly the set of indispensable binomials. These binomials are specified by their weighted graph embeddings which are depicted in Table \ref{table dep1}.

\section{Depth 1 diagrams}\label{depth1}

    \subsection{Special Binomials}
    
    While computing the reduced Gr\"obner basis of a code is generally very difficult, codes arising from Euler diagrams that fall in certain classes have toric ideals that are easy to describe. In this case, there are certain binomials which are required to be present in a generating set of the ideal.
    More precisely, a binomial $f$ is called \emph{indispensable} if, for any set $\mathcal{B}$ of binomial generators of the ideal, $f \in \mathcal{B}$ or $-f \in \mathcal{B}$. 
    We will introduce several special subgraphs that naturally give rise to the indispensable binomials.
    
    \begin{table}[h]
        \centering
        \begin{tabular}{|  >{\centering\arraybackslash}m{.5in}| *2{>{\centering\arraybackslash}m{1.75in}|} >{\centering\arraybackslash}m{1.25in}|   @{}m{0pt}@{}}
        \hline
        Type & Binomial & Euler diagram & Dual graph 
        \\ \hline 
    1    &  
    $t_{\{1\}}t_{\{2\}} - t_{\{1,2\}}$    &  \includegraphics[page=2,scale=.6]{lozenge.pdf} & 
    \includegraphics[page=1,scale = .5]{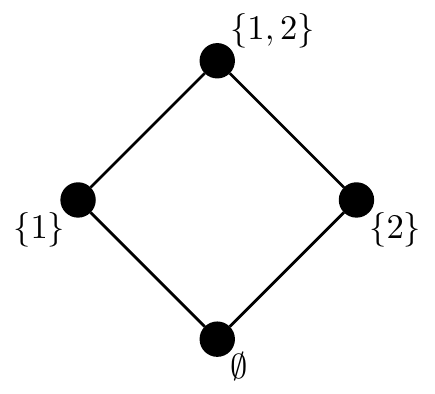}
        \\ \hline
    2 & $t_{\{2\}}t_{\{1,3 \}} - t_{\{1,2,3\}}$ &
    \includegraphics[page=2,scale=.6]{domino.pdf} 
     & 
    \includegraphics[page=1,scale = .5]{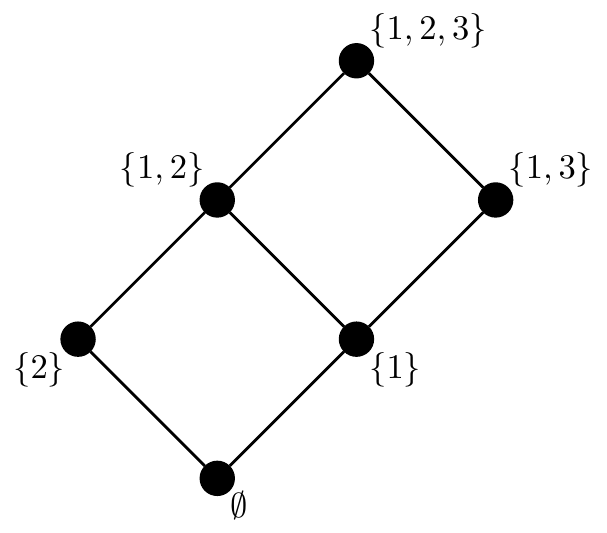}  
    \\ \hline
    3 & $t_{\{1,2\}}t_{\{1,3 \}} - t_{\{1\}}t_{\{1,2,3\}}$ &
    \includegraphics[page=2,scale = .6]{lollipop.pdf} 
    &
    \includegraphics[page=1,scale = .5]{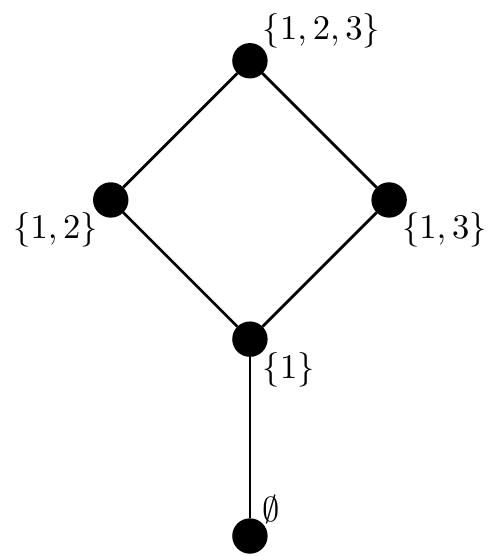}
    \\ \hline
    4 
    & $t_{\{1,2,3\}}t_{\{1,4 \}} - t_{\{1,2\}}t_{\{1,3,4\}}$ 
    &
    \includegraphics[page=2,scale = .6]{flower.pdf} 
    &
    \includegraphics[page=1,scale = .5]{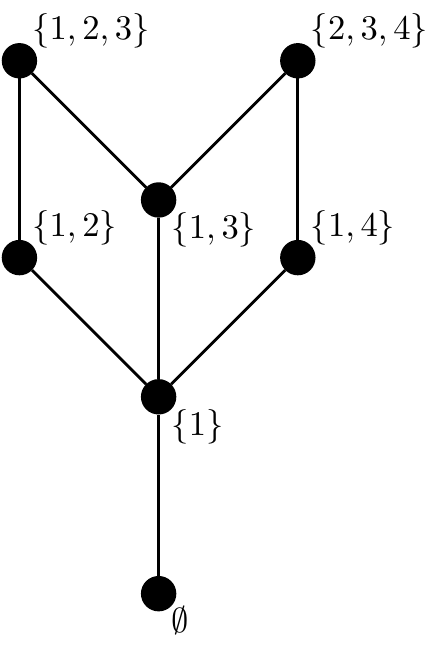}
    \\ \hline
    5
    & $t_{\{1,2,3\}}t_{\{1,4 \}} - t_{\{1,2\}}t_{\{1,3,4\}}$ 
    &
    \includegraphics[page=4,scale = .6]{flower.pdf} 
    &
    \includegraphics[page=3,scale = .5]{flower.pdf}
    \\ \hline
    6
    & $t_{\{1,2,3\}}t_{\{1,4 \}} - t_{\{1,2\}}t_{\{1,3,4\}}$ 
    &
    \includegraphics[page=6,scale = .6]{flower.pdf} 
    &
    \includegraphics[page=5,scale = .5]{flower.pdf}
    \\ \hline
        \end{tabular}
        \caption{Indispensable binomials for 1 inductively pierced, depth 1 binomials}
        \label{table dep1}
    \end{table}
    
    In the following table we list several  diagrams and their associated dual graphs. Any 1 inductively pierced  diagram with depth less than or equal to 1 has indispensable binomials that are determined in Table \ref{table dep1}.
    
    \begin{thm}\label{thm dep1 1p}
        Let $\cC$ be a neural code such that its associated diagram $\cD$ is a one inductively pierced diagram of depth $\leq 1$. Let  $ G = (\cV,\cE,\mu)$ be the  dual graph associated to $\cD$. If $\wtil \cD$ is a diagram such that it's  dual graph $ H = (\cW,\cF,\nu) $ belongs to Table \ref{table dep1} and has an embedding $\phi$ into $G$, then the image of the associated binomial under $\phi$ is an indispensable binomial for the ideal of $\pi_{\cC}$. Moreover, $\pi_{\cC}$ has no other indispensable binomials.  
    \end{thm}
     
     The theorem is an easy consequence of Lemmas \ref{Lemma wg 23} -- \ref{Lemma wg 6}.
    
    \subsection{Binomials}

    \begin{lemma}\label{Lemma wg 23}
      The toric ideal of a diagram $\cD$ has an indispensable binomial of weight 2 and weight 3 respectively if and only if there is a weighted graph embedding from a Type 1 graph (respectively Type 2 graph) to the weighted dual graph of $\cD$.  
    \end{lemma}
      \begin{proof}
         Monomials of weight 2 arises only in the form $t_{\{1\}}t_{\{2\}}$ and $ t_{\{1,2\}}$. Thus the only binomials of order 2 are of the form $t_{\{1,2\}} - t_{\{1\}}t_{\{2\}}$.
      
         Monomials of weight 3 arise only as $t_{\{1,2,3\}} $ or $t_{\{1,2\}}t_{\{3\}}$ or higher order terms. These diagrams arise only from   a `stacked lozenge' diagram. 
         
         Notice the alternative of a pair of monomials $t_{\{1,2\}}t_{\{3\}}$ and $t_{\{1\}} t_{\{2,3\}}$. The existence of both these monomials implies that $t_{\{2\}}$ exists. Then the binomial  $t_{\{1,2\}}t_{\{3\}}-t_{\{1\}} t_{\{2,3\}}$ is generated by binomials $t_{\{1,2\}} - t_{\{1\}} t_{\{2\} }$ and $t_{\{2,3\}} - t_{\{2\}} t_{\{3\} }$. 
      \end{proof}

    \begin{lemma}\label{Lemma wg 4}
      The toric ideal of a diagram $\cD$ has an indispensable binomial of weight 4 if and only if there is a weighted graph embedding from a Type 3 graph to the weighted dual graph of $\cD$.  
    \end{lemma}
    \begin{proof}
      Monomials of order 4 arise from terms $t_{\{i,j\}}t_{\{k,l\}}$ and  $t_{\{i,j,k\}} t_{\{l\}}$.
      
      Let us begin with monomials of the form $t_{\{i,j\}}t_{\{k,l\}}$.
      We rule out the weight 2 + 2  term with $i = k,\ j =l$, ie $t_{\{1,2\}}t_{\{1,2\}}$,  as it has no linear or quadratic balancing monomials. Now consider  terms of the form $t_{\{1,2\}}t_{\{1,3\}}$. $\{1,2\},\{1,3\} \in \cC $  imply $\{1\} \in \cC$.  The only possible nontrivial balancing monomial is $t_{\{1,2,3\}} t_{\{1\}} $, which requires zone $\{1,2,3\} $ exists. In this case we have $p=t_{\{1,2\}}t_{\{1,3\}} -t_{\{1,2,3\}} t_{\{1\}}  $ in the kernel. If in addition $\{2\} \in \cC$, then  $p$ is generated by binomials $t_{\{1,2\}} - t_{\{1\}} t_{\{2\}} $ and $t_{\{1,2,3\}} - t_{\{1,3\}}t_{\{2\}}$. If neither $\{2\}$ or $\{3\}$ exists then $U_2\cup U_3 \subset U_1$, so the binomial arises from the lollipop diagram. 
      Finally consider monomial $t_{\{1,2\}} t_{\{3,4\}} $ and note that the binomial  $t_{\{1,2\}} t_{\{3,4\}} - t_{\{2,3\}} t_{\{1,4\}} $  is  not permitted in a one piercing diagram. 
      If $\{1,2,3\},\{4\} \in \cC$ then we have binomial  $ p = t_{\{1,2\}} t_{\{3,4\}} - t_{\{1,2,3\}} t_{\{4\}} $ but then $\{3\}\in \cC$ so $ p $ is generated by $t_{\{3,4\}} - t_{\{3\}}t_{\{4\}}$ and $t_{\{1,2,3\}} - t_{\{1,2\}} t_{\{3\}}$.
      This concludes all possible binomials containing a weight 2 + 2 term.
      
      Now we consider terms of the form $t_{\{i,j,k\}} t_{\{l\}}$. We only need to consider binomials with balancing terms which are weight 3 + 1. The possibilities are $t_{\{1,2,3\}} t_{\{4\}}$ or $t_{\{1,2,3\}} t_{\{1\}}$. There are no alternative possible balancing weight 4 = 3+1 monomials  for the second type. For the first type the only balancing term is $t_{\{2,3,4\}} t_{\{1\}}$, clearly this does not exist if $U_1 \subset U_2 \cup U_3 $. We may assume that $U_3 \subset U_1$, and $U_3 \cap U_2 \neq \emptyset$. But zones if zones $\{1\}$, $\{1,2,3\}$, $\{2,3,4\}$, and $\{4 \} $ exist this requires adding a lozenge to the stick-and-lozenge diagram with boundary incident on two sides.
     \end{proof}

    \begin{lemma}\label{Lemma wg 5}
      The toric ideal of a diagram $\cD$ has an indispensable binomial of weight 5 if and only if there is a weighted graph embedding from a Type 4,5,  or 6 graphs to the weighted dual graph of $\cD$.  
    \end{lemma}
    \begin{proof}
       Monomials of order 5 arise only as weight 3+2 terms: $t_{\{i,j,k\}}t_{\{l,m\}}$. We will write these as $ t_{\{1,2,3\}}t_{\{l,m\}}$ with $ U_3 \subset U_1$.
       
       Note that, if $U_3 \subset U_1$, $ t_{\{1,2,3\}}t_{\{2,3\}}$ does not exist, and $t_{\{1,2,3\}}t_{\{1,m\}}$ for $m = 2,3$ has no balancing term. 
       
       Let us consider $ t_{\{1,2,3\}}t_{\{1,4\}}$. If $U_4 \cap U_2 \neq \emptyset  $ then $U_4 \subset U_1 $, the only permitted balancing term is  $ t_{\{1,2,4\}}t_{\{1,3\}}$ this arises only as 2 one-piercings within $U_1$. Now let us consider $ t_{\{1,2,3\}}t_{\{2,4\}}$ the only possible balancing monomials are $t_{\{1,2\}}t_{\{2,3,4\}}$ or $t_{\{2,3\}}t_{\{1,2,4\}}$ however neither zones $\{2,3\}$ nor $\{2,3,4\}$ are permitted.
       
       Finally let us consider $a = t_{\{1,2,3\}}t_{\{4,5\}}$, but $(U_4\cup U_5 )\cap U_i$ can only be non-empty for one of $i = 1,2$. If $(U_4\cup U_5 )\cap U_1$ is nonempty,  we have zones $\{1,4,5\}$ or $\{1,4\}$, however the zones $\{2,3\}$ and $\{2,3,5\}$ do not exist so these do not correspond to balancing monomials for $a$. On the other hand if  $(U_4\cup U_5 )\cap U_2$ is non-empty and $U_5 $ is contained in $U_4$ and is a piercing of $U_2$ then we have the binomial $  t_{\{1,2,3\}}t_{\{4,5\}} - t_{\{2,4,5\}}t_{\{1,3\}} $, but this is binomial is generated by the pair $t_{\{1,2,3\}} -t_{\{1,3\}}t_{\{2 \}}$ and  $t_{\{2,4,5\}} -t_{\{4,5\}}t_{\{2 \}}.$
    \end{proof}

    \begin{lemma}\label{Lemma wg 6}
      The toric ideal of any Euler diagram $\cD$ has no indispensable binomials of order 6 or higher. 
    \end{lemma}
    \begin{proof}
       Weight 6 monomials arise only as weight  3+3 terms. 
       If the zones have 2 curves in common $\{1,2,3\}$ and  $\{1,2,4\}$, then a balancing monomial would contain  zone-variable corresponding to zone $\{1,2,3,4\}$ which does not exist.  
       Suppose the zones have 1 curve in common $\{1,2,3\}$ and $\{1,4,5\}$, then $U_1 $ cannot be contained in $U_2\cup U_3$. If zone $\{1,3,5\}$ exists then $U_3\cup U_5 \subset U_1$ but then the zone $\{1,4,5\}$ only exists if $U_4\subset U_1$ so the balancing term is $t_{\{1,3,5\}}t_{\{1,2,4\}}$   Again, if the zone  $\{1,2,4\}$ exists the zone $\{1,3,5\}$ does not exist as loops are prohibited in the piercing graph.
       Finally, suppose the two zones have no curves in common $\{1,2,3\}$ and $\{4,5,6\}$. We must have $U_3 \subset U_1$ and $U_6 \subset U_4$. But if the zone $\{2,4,6 \}$ exists, the zone $\{1,3,5\} $ is prohibited as this would violate the piercing construction.
    \end{proof}

    \section{External Diagrams}
    
    \begin{defn} Let $\cD$ be a well-formed Euler diagram on curves $ \{ \lambda_1, \ldots, \lambda_n \} $ with corresponding interiors $ \{ U_1, \ldots, U_n\} $. 
    If 
    \[
        U_i \setminus \bigcup_{j \neq i} U_j \neq \emptyset
    \]
    for each $i \in [n]$, then $D$ is called an \emph{external} Euler diagram.
    If a code has an external Euler diagram as a realization, then we call the code \emph{external} as well.
    \end{defn}

    \begin{figure}[h]
    \center
    \includegraphics[width=4in]{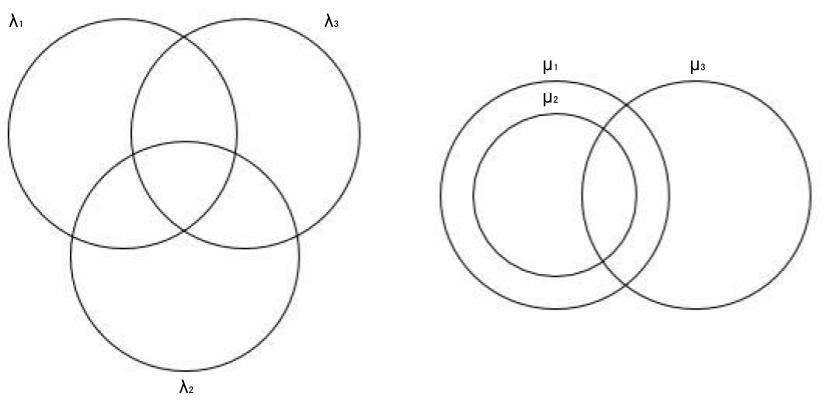}
    \caption{An external Euler diagram (left) and an Euler diagram that is not external (right).}\label{fig:externals}
    \end{figure}
    
    In Figure~\ref{fig:externals}, the diagram on the left is external, as none of the $\lambda_i$ are completely contained in the interior of the others. However, the diagram on the right is not external, as $\mu_2$ is contained within the interior of $\mu_1$. 

    We point out here that external diagrams on $n$ neurons will always induce a code containing the $n$ codewords where all but one entry is zero.
    
    \begin{defn}
        Let $\cC$ be an external code and $c \in \cC$.
        The \emph{support} of $c = c_1\dots c_n$, which we denote $\supp(c)$, is the set
        \[
            \supp(c) = \{ i \mid c_i \neq 0\}.
        \]
        The \emph{weight} of $c$, which we denote $\wt(c)$, is $\wt(c) = |\supp(c)|$.
    \end{defn}
    
    It is easy to see then that, for an external Euler diagram, any codeword can be written as the sum of codewords with weight one. This gives us a set of nice binomials that we know must be in the toric ideal of any external diagram.
    
    \begin{example}\label{ex: example GB}
        Consider Figure~\ref{fig:externals}. The corresponding code of this diagram is 
        \[
            \cC_1 = \{000, 100 , 010, 001, 110, 101, 011, 111 \}.
        \]
    Two reduced Gr\"obner bases of the toric ideal $I_{\cC_1}$ are
    \[
    \begin{aligned}
        G_1 = & \{ t_{\{1,3\}} t_{\{2,3\}} - t_{\{3\}}t_{\{1,2,3\}},
        t_{\{1,2\}} t_{\{2,3\}} - t_{\{2\}} t_{\{1,2,3\}}, 
        t_{\{1\}} t_{\{2,3\}} - t_{\{1,2,3\}}, 
        t_{\{1,2\}} t_{\{1,3\}} - t_{\{1\}} t_{\{1,2,3\}}, \\
        & t_{\{2\}} t_{\{1,3\}} - t_{\{1,2,3\}},
        t_{\{3\}} t_{\{1,2\}} - t_{\{1,2,3\}}, 
        t_{\{2\}} t_{\{3\}} -  t_{\{2,3\}}, 
        t_{\{1\}} t_{\{3\}} - t_{\{1,3\}}, 
        t_{\{1\}} t_{\{2\}} - t_{\{1,2\}}\},
    \end{aligned}
    \]
    for which the grevlex ordering is used, and
    \[
        G_2 = \{t_{\{2,3\}}  - t_{\{2\}} t_{\{3\}} , t_{\{1,3\}}  - t_{\{1\}} t_{\{3\}} , t_{\{1,2\}}  - t_{\{1\}} t_{\{2\}} , t_{\{1,2,3\}}  - t_{\{1\}} t_{\{2\}} t_{\{3\}}\},
    \]
    for which the weight vector $(0,0,0,1,1,1,2)$, and ties decided by the grevlex ordering, is used.
    The only binomials that are in both $G_1$ and $G_2$ are $t_{\{1,2\}} - t_{\{1\}} t_{\{2\}}, t_{\{1,3\}} - t_{\{1\}} t_{
    \{3\}}$, and $t_{\{2,3\}} - t_{\{2\}} t_{\{3\}}$. 
    It turns out that this is not a coincidence, as they are indispensable binomials. 
    \end{example}
    
    Given a code $\cC$, let 
    \[
        \calA_\cC = \{t_{w^{-1}(c_i)} t_{w^{-1}(c_j)} - t_{w^{-1}(c_k)} \mid c_i,c_j,c_k \in \cC, \, c_i + c_j = c_k, \, \wt(c_i)=\wt(c_j)=1\}.
    \]
    
    \begin{prop}\label{prop: A indispensable}
        If $\cC$ is an external code, then the binomials in $\calA$ are indispensable.
    \end{prop}
    
    \begin{proof}
        Let $b \in \calA_\cC$. 
        If $b$ does appear in a binomial generating set $G$ of $I_\cC$, then $b$ is a polynomial combination of other binomials in $G$.
        However, one of the monomials of $b$ has degree $1$, and the toric ideal has no elements of degree $0$.
        So, any product of binomials in $G$ must result in a     polynomial whose terms all are of degree $2$ or higher.
        Therefore, $b$ must be indispensable.
    \end{proof}
    
    As a result, we know that for any external code, the binomials in $\calA_\cC$ will, up to sign, appear in every reduced Gr\"obner basis of $I_\cC$. For the toric ideal of the code $\cC_1$ from Example~\ref{???}, it is unsurprising that $\calA$ is exactly the set of indispensable binomials. However, it is not obvious that these two sets of binomials coincide for all external codes.

    Let $\omega$ be the weight vector on $\CC[t_{w^{-1}(c_1)} , \ldots , t_{w^{-1}(c_n)}]$ satisfying $\omega_i = \wt(c_i) - 1$. Define $\prec_\omega$ to be the monomial ordering on $\CC [ t_{w^{-1}(c_1)} , \ldots , t_{w^{-1}(c_n)} ]$ where $t^{\alpha} \prec_\omega t^{\beta}$ if $\omega \cdot \alpha < \omega \cdot \beta$, with ties being determined by the grevlex order.
    Recall as well that $c_0 = 0\dots0$.
    
    \begin{lem}\label{lem: lead terms}
        If $c_i \in \cC$ and $c_i = \sum_{c \in \cC'} c$ for some subset $\cC'$ of $\cC \setminus \{c_0,c_i\}$, then $\prod_{c \in \cC'} t_{w^{-1}(c)} \prec_\omega t_{w^{-1}(c_i)}$. 
    \end{lem}
    
    \begin{proof}
        Let $c_i \in \cC$ with $\wt(c_i) = k$. Assume that $c_i = \sum_{c \in \cC'} c$ for some subset $\cC'$ of $\cC$ with no $c = c_i$ and $|\cC'| = m \geq 2$. Then $\sum_{c \in \cC'} \wt(c) = k$ and we see that $\omega \cdot t_{w^{-1}(c_i)} = k-1$. Moreover, 
        \[
            \omega \cdot \prod_{c \in \cC'} t_{w^{-1}(c)} = \sum_{c \in \cC'} \omega(c) \leq k - |\cC'| = k-m < k =  \omega \cdot t_{w^{-1}(c_i)},
        \]
        as desired.
    \end{proof}
    
    Since we are concerned with only external diagrams, we know that any $c_i \in \cC$ with $\wt(c_i) \geq 2$ is the sum of some set of $c_j \in \cC$ with $\wt(c_j) = 1$. This implies that we have polynomials in our toric ideal with one linear term $t_{w^{-1}(c_i)}$ and one term with degree $\wt(c_i)$.
    
    \begin{lem}\label{lem: create B}
        Let $\cC$ be an external code.
        The reduced Gr\"obner basis of $I_{\cC}$ with respect to $\prec_\omega$ is exactly the set 
        \[
            B = \left\{ t_{w^{-1}(c)} - \prod_{j \in \supp(c)} t_{w^{-1}(e_j)} \mid \wt(c) = k \geq 2\right\},
        \]
        where $e_j$ is the $j^{th}$ standard basis vector.
    \end{lem}
    
    \begin{proof}
        Let $c \in \cC$ with $\wt(c) = k \geq 2$, and set 
        \[
            g_c = t_{w^{-1}(c)} - \prod_{j \in \supp(c)} t_{w^{-1}(e_j)}.
        \]
        By construction, $g_i \in I_{\cC}$. By Lemma~\ref{lem: lead terms}, we know that $t_{w^{-1}(c)}$ will be the initial term of this binomial with respect to $\prec_\omega$. So, for all $c \in \cC$ with $\wt(c) \geq 2$, we have $t_{w^{-1}(c)} \in \initp(I_{\cC})$.
    
        Let $\cG$ be the reduced Gr\"obner basis of $I_{\cC}$ with respect to $\prec_\omega$. So, for all $c \in \cC$ with $\wt(c) \geq 2$, there exists $g \in G$ such that $\initp(g) = t_{w^{-1}(c)}$. Also, for all $c' \in \cC$ with $\wt(c') = 1$, there does not exist $p \in I_{\cC}$ with $\initp(p) = t_{w^{-1}(c')}$, since, also by construction, there can be no binomial $t_{w^{-1}(c')} - t_{w^{-1}(c'')}$ in $I_{\cC}$ where $\wt(c') = \wt(c'') = 1$.
        
        Let $g \in \cG$ with $\initp(g) = t_{w^{-1}(c)}$ for some $c \in \cC$ with $\wt(c) \geq 2$. Since $\cG$ is a reduced Gr\"obner basis, we know that no initial term of any $g' \in \cG$, $g' \neq g$, can divide either term of $g$. But for all $c' \in \cC$ with $\wt(c') \geq 2$, $t_{w^{-1}(c')} \in \initp(I_\cC)$. So, the nonlinear term of $g$ must be the product of some $t_{w^{-1}(c_l)}$ with $\wt(c_l) = 1$. 
        This forces $g = g_c$.
        Moreover, there can be no binomials $g'$ for which $\initp(g')$ has degree at least $2$, since $\initp(g')$ would then be divisible by $\initp(g_c)$ for some $c$.
        Thus, $\cG = \calB$.
    \end{proof}
    
    While this lemma gives us a reduced Gr\"obner basis of $I_\cC$, not all of the binomials are indispensable, as evidenced by Example~\ref{ex: example GB}
    
    \begin{lem}\label{lem: elements of B in A}
        The only binomials from $\calB$ in the reduced Gr\"obner basis of $I_{\cC}$ with respect to grevlex are the binomials in $\calA_\cC$.
    \end{lem}
    
    \begin{proof}
        Choose an order $c_1, \ldots , c_n$ of the codewords in $\cC$ such that $\wt(c_i) < \wt(c_j)$ implies $i < j$. Let $\cG$ be the reduced Gr\"obner basis with respect to grevlex. We have shown that the set $A$ is indispensable, so we have that $\calA_\cC \subseteq G$.
        
        Consider a binomial $g \in \calB \setminus \calA_\cC$. Then 
        \[
            g = t_{w^{-1}(c)} - \prod_{j \in \supp(c)} t_{e_j}
        \]
        for some $c \in \cC$ with $\wt(c) = k \geq 3 $. Since $\cC$ is external, we know that for some $t_{e_i},t_{e_j}$ in the nonlinear term of $g$, $g' = t_{c'}-t_{e_i} t_{e_j}$ is in $\calA_\cC$ for some $c' \in \cC$. So, $g_1 \in G$. But, $\initp(g')$ divides $\initp(g)$, so $g \notin \cG$. So no element of $\calB$ that is not in $\calA_\cC$ is indispensable.
    \end{proof}
    
    \begin{thm}
        For an external code $\cC$, the indispensable binomials of $I_{\cC}$ are exactly those in $\calA_\cC$.
    \end{thm}
    
    \begin{proof}
        This follows directly from Lemmas~\ref{prop: A indispensable}, \ref{lem: create B}, and \ref{lem: elements of B in A}.
    \end{proof}

So for any given external diagram, we know exactly which binomials must be in a reduced Gr\"obner basis of its toric ideal. However, other binomials clearly appear in some reduced Gr\"obner bases of the toric ideals of external diagrams. For certain classes of external diagrams, there is even more to say about Gr\"obner bases. 

To close this section, we will use graphs to help us describe properties of toric ideals of external codes. 
Given an external code $\cC$ on neurons $\lambda_1,\dots,\lambda_k$, let $\Delta_\cC$ denote the graph with vertices $1,\dots, k$ and edges $\{i, j\}$ if $e_i+e_j$ is a codeword in $\cC$. 

\begin{lem}\label{lem: distance two}
    Let $\cC$ be a code such that $\Delta_\cC$ is a tree.
    If $i,j$ are vertices of $\Delta_\cC$ that are distance two apart, then there exists a unique vertex $k$ for which $t_{w^{-1}(c_i)}t_{w^{-1}(c_j)+w^{-1}(c_k)} - t_{w^{-1}(c_i)+w^{-1}(c_k)}t_{w^{-1}(c_j)} \in I_\cC$
\end{lem}

The proof of this lemma is short and straightforward, so its proof is omitted.
For the next result, if $v$ is the vertex of a graph, let $d(v)$ denote the degree of $v$.

\begin{thm}
    Let $\cC$ be an external code such that $\Delta_\cC = (V,E)$ is a tree.
    In the universal Gr\"obner basis of $I_\cC$, there are $\sum_{v \in V} \binom{d(v)}{2}$ polynomials of the form $t_{w^{-1}(c_1)}t_{w^{-1}(c_2)} - t_{w^{-1}(c_3)}t_{w^{-1}(c_4)}$ for $c_1,\dots,c_4 \in \cC$ and $c_1+c_2=c_3+c_4$.
\end{thm}

\begin{proof}
Suppose $\Delta_\cC$ is a tree, and let $i, j,k$ be vertices of $\Delta_\cC$ such that $k$ is adjacent to both $i$ and $j$.
By Lemma~\ref{lem: distance two}, we know that $p(t) = t_{w^{-1}(c_i)}t_{w^{-1}(c_j)+w^{-1}(c_k)} - t_{w^{-1}(c_i)+c_k}t_{w^{-1}(c_j)} \in I_\cC$.
Without loss of generality let $\initp(p(t)) = t_{w^{-1}(c_i)}t_{w^{-1}(c_j)+w^{-1}(c_k)}$. 

Now, we will show that $p(t)$ is in some reduced Gr\"obner basis $\cG$ of $I_\cC$.
Consider the grevlex order.
In this case, no binomial in $\cG$ has a linear initial term. 
Thus, the only way for $\initp(p(t))$ to be divisible by an initial term of a binomial in $\cG$ is if that term is $\initp(p(t))$ itself.
So, $\initp(p(t))$ is the initial term of a binomial $b$ in $\cG$.
Since $\pi_\cC(\initp(p(t))) = x^ix^jx^k$,
there are three possibilities of the non-initial term of $b$: $t_{w^{-1}(c_i)+c_j+c_k}$, $t_{w^{-1}(c_i)}t_{w^{-1}(c_j)}t_{w^{-1}(c_k)}$, and $t_{w^{-1}(c_i)+c_k}t_{w^{-1}(c_j)}$.
The first possibility cannot happen since $\Delta_\cC$ is a tree, meaning $|\supp(t_{w^{-1}(c)})| \leq 2$ for all $c \in \cC$. The second possibility also cannot occur since, otherwise, it would be the initial term of $b$ under grevlex.
This leaves one possibility, hence $p(t) \in \cG$.

Since we obtain a polynomial $p(t)$ for each vertex in $\Delta_\cC$ that is the midpoint of a length-two path, the number of homogeneous quadratic binomials in the universal Gr\"obner basis is the same as the number of paths in the tree of length two, which are enumerated by
\[
\sum_{v\in V}\binom{d(v)}{2}. \qedhere
\]
\end{proof}

\section{Internal Codes}

In this section, we focus on a particular class of codes.
Let the $n^{th}$ \emph{internal code} be the code $\cL_n = \{0\dots0,c_1,\dots,c_{2n}\}$ where the nonzero codewords are 
\[
    c_j = \begin{cases}
            e_1 + \sum_{i=2}^j e_j & \text{ if } j \leq n, \\
            c_{j-n}-e_1 & \text{ if } j > n.
        \end{cases}
\]
It is clear to see that this code has a corresponding Euler diagram
\begin{center}
	\begin{tikzpicture}
		\draw (-3.5,-3.5) -- (3.5,-3.5) -- (3.5,3.5) -- (-3.5,3.5) -- cycle;
		\draw (0,0) circle [radius=3];
		\draw (-1,0) circle [radius=1.75];
		\draw (1,0) circle [radius=1.75];
		\draw (0.75,0) circle [radius=1.25];
		\draw [dashed] (0.7,0) circle [radius=1.12];
		\draw [dashed] (0.575,0) circle [radius=0.875];
		\draw (0.5,0) circle [radius=0.75];
		
		\node at (-2.25,0) {\small{$\lambda_1$}};
		\node at (-2.5,2.5) {\small{$\lambda_2$}};
		\node at (2.45,0) {\small{$\lambda_3$}};
		\node at (1.66,0) {\tiny{$\cdots$}};
		\node at (1.05,0) {\small{$\lambda_n$}};
	\end{tikzpicture}
\end{center}

Call a binomial $t^a - t^b \in I_\cC$ \emph{primitive} if there is no binomial $t^u - t^v \in I_\cC$ such that both $t^u$ divides $t^a$ and $t^v$ divides $t^b$. 
The \emph{Graver basis} of $I_\cC$ is the set of all primitive binomials in $I_\cC$. 
Since every binomial in a reduced Gr\"obner basis of a toric ideal is primitive, the Graver basis will contain the universal Gr\"obner basis of $I_\cC$. 
In fact, in certain cases, the Graver basis is identical to the universal Gr\"obner basis.

\begin{defn}
    Let $A$ be a $k\times n$ matrix.
    Its \emph{Lawrence lifting} is
    \[
        \Lambda(A) = \begin{bmatrix}
            A & 0_{k \times n} \\
            I_n & I_n
        \end{bmatrix}
    \]
    where $0_{k \times n}$ is the $k\times n$ zero matrix and $I_n$ is he $n\times n$ identity matrix.
\end{defn}

Let $\cC = \{0\dots0,c_1,\dots,c_k\}$ be a code.
For notational convenience, let $M_\cC$ denote the matrix with columns $c_1,\dots,c_k$. 
Thus, we can think of the toric ideal $I_\cC$ as the toric ideal of either the code $\cC$ or of the matrix $M_\cC$.
Because row operations on matrices preserve linear dependencies, we can apply them to $M_{\cL_n}$ and compute the toric ideal of the simpler matrix.
In our case, it is straightforward to verify that $M_{\cL_n}$ is row-equivalent to the Lawrence lifting of the $1 \times n$ matrix $\begin{bmatrix} 1 & \cdots & 1 \end{bmatrix}$, say by multiplying $M_{\cL_n}$ on the left by the matrix with rows $r_1,\dots,r_n$, where
\[
    r_i = 
        \begin{cases}
            e_1 & \text{ if } i = 1, \\
            e_{i-1}-e_i & \text{ if } 2 \leq i \leq n-1, \\
            e_n & \text{ if } i = n.
        \end{cases}
\]

\begin{thm}[{\cite[Theorem~7.1]{sturmfels}}]\label{thm: graver equal universal}
	Let $\cC$ be any combinatorial neural code.
	The following sets are identical:
	\begin{enumerate}
		\item the universal Gr\"obner basis of $I_{\Lambda(M_\cC)}$,
		\item the Graver basis of $I_{\Lambda(M_\cC)}$,
		\item the minimal binomial generating set of $I_{\Lambda(M_\cC)}$.
	\end{enumerate}
\end{thm}

Since $M_{\cL_n}$ is row-equivalent to $A_n = \Lambda(\begin{bmatrix} 1 & \cdots & 1 \end{bmatrix})$, all of the results of the preceding theorem hold for $I_{\cL_n}$ as well.
For convenience, we will continue by considering $I_{A_n}$.

To prove the main theorem of the section, let 
\[
    \cU_n = \{t_{\{1,j\}}t_{\{k\}}-t_{\{1,k\}}t_{\{j\}} \mid 2 \leq j < k \leq n, j \neq k \}.
\]
It is clear that $\cU_n \subseteq I_{\cL_n}$, by verifying that
\[
    \pi_{\cL_n}(t_{\{1,j\}}t_{\{k\}}-t_{\{1,k\}}t_{\{j\}}) = 0.
\]

\begin{lemma}\label{lem: all quads}
    For all $n$, $\cU_n$ contains all monic homogeneous quadratic binomials in $I_{A_n}$, up to sign.
\end{lemma}

\begin{proof}
First, note that the first $n$ columns and the last $n$ columns of $A_n$ are linearly independent sets.
So, if $t^a-t^b \in I_{A_n}$ is monic, homogeneous, and quadratic, then each monomial must be a product of one variable $t_{\{1,j\}}$ and one variable $t_{\{k\}}$. 
If $j<k$ in the leading term, then the only possible binomial is $t_{\{1,j\}}t_{\{k\}}-t_{\{1,k\}}t_{\{j\}}$. 
Similarly, if $k>j$, then the only possible binomial is $-(t_{\{1,j\}}t_{\{k\}}-t_{\{1,k\}}t_{\{j\}})$. 
The case $j=k$ cannot happen, since the vector $e_1+2e_j$ can only be obtained as the sum of columns of $A_n$ in one way.
\end{proof}

A result due to Gross, Obatake, and Youngs says if a diagram is 1-inductively pierced then its toric ideal is generated by quadratics \cite{GrossObatakeYoungs}. These quadratic binomials are homogeneous based on a lemma by Sturmfels \cite{sturmfels}. Since we have found all possible forms of quadratic binomials we can now say that the set $B$ generates the toric ideal.

\begin{lemma}\label{lem: Un primitive}
    For all $n$, the binomials of $\cU_n$ are primitive.
\end{lemma}

\begin{proof}
    Let $t_{\{1,j\}}t_{\{k\}}-t_{\{1,k\}}t_{\{j\}} \in \cU_n$, and suppose there is some binomial $t^v-t^w \in I_{A_n}$ such that $t^v$ divides $t^a$ and $t^w$ divides $t^b$.
    If $\deg(t^v) = 1$, then $\deg(t^w) = 1$, but this cannot happen since no column of $A_n$ is a scaling of another column. 
    So, $t^v = t_{\{1,j\}}t_{\{k\}}$.
    Again by the structure of $A_n$, this forces $t^v = t_{\{1,k\}}t_{\{j\}}$.

    As stated above, we know that the toric ideal is generated by quadratic homogeneous binomials. 
    Since we know this, we can guarantee that the degree of any monomial in the toric ideal is at least 2. 
    Because of this, we know that the only way to construct the binomial $t^v-t^w$ such that $t^v$ divides $t^a$ and $t^w$ divides $t^b$ is if $t^v$ and $t^w$ both are at least of degree 2. 
    If $t^v$ and $t^w$ are both at least of degree 2 then $t^a=t^v$ and $t^b=t^w$. 
    Therefore, all of the binomials in the set $B$ are primitive binomials.
    This implies that the set $B$ is a minimal binomial generating set of $I$ and is therefore the Graver basis of $I_{A_n} = I_{\cL_n}$ and the universal Gr\"obner basis of $I_{\cL_n}$ by Sturmfels's Theorem \cite{sturmfels}. 
\end{proof}    
	
\begin{thm}
    The universal Gr\"obner basis of $\cL_n$ is $\cU_n$.
\end{thm} 

\begin{proof}
    By Theorem~\ref{thm: GOY}, part $3$, we know that $I_{\cL_n}$ is generated by quadratics.
    Since $I_{\cL_n}$ is homogeneous, these quadratics must be homogeneous.
    By Lemma~\ref{lem: all quads} and Lemma~\ref{lem: Un primitive}, $\cU_n$ is a minimal binomial generating set of $I_{\cL_n}$.
    Therefore, by Theorem~\ref{thm: graver equal universal}, $\cU_n$ is the universal Gr\"obner basis of $\cL_n$.
\end{proof}

\end{document}